\documentclass[12pt,a4paper]{amsart}

\usepackage{amsmath, amsfonts, xifthen, latexsym, amssymb, amsthm, amscd}

\usepackage[utf8]{inputenc}
\usepackage[shortlabels]{enumitem}
\usepackage{graphicx}
\usepackage{url}
\usepackage{hyperref}
\hypersetup{colorlinks=true,citecolor=blue,filecolor=blue,linkcolor=blue,urlcolor=blue}
\usepackage[margin=1.4in]{geometry}

\newtheorem{theorem}{Theorem}

\newtheorem*{theorem*}{Theorem}
\newtheorem{lemma}[theorem]{Lemma}

\newtheorem*{lemma*}{Lemma}

\newtheorem{proposition}[theorem]{Proposition}

\newtheorem{corollary}[theorem]{Corollary}
\newtheorem{cory}[theorem]{Corollary}

\theoremstyle{definition}
\newtheorem{definition}[theorem]{Definition}

\newtheorem{remark}[theorem]{Remark}
\newtheorem{remarks}[theorem]{Remarks}

\newcommand{\sta}{{{}^\ast}}
\newcommand{\al}{{\alpha}}

\newcommand{\be}{{\beta}}

\newcommand{\eps}{{\varepsilon}}
\newcommand{\de}{{\delta}}
\newcommand{\De}{{\Delta}}
\newcommand{\ga}{{\gamma}}
\newcommand{\Ga}{{\Gamma}}

\newcommand{\la}{{\lambda}}

\newcommand{\si}{{\sigma}}

\renewcommand{\phi}{{\varphi}}

\newcommand\map{\operatorname{Map}}

\newcommand\im{\operatorname{im}}

\newcommand{\Id}{\operatorname{Id}}

\renewcommand{\ge}{\geqslant}
\renewcommand{\le}{\leqslant}

\newcommand\Z{\mathbb Z}
\newcommand\C{\mathbb C}
\newcommand\N{\mathbb N}

\newcommand{\into}{\hookrightarrow}

\newcommand{\cal}[1]{{\mathcal #1}}
\newcommand{\ov}{\overline}

\usepackage{etoolbox}
\newtoggle{no_cases}\toggletrue{no_cases}
\newcommand{\case}[2][]{\iftoggle{no_cases}{\left\{\begin{array}{ll}#2 & #1}{\\#2 & #1}\togglefalse{no_cases}}
\newcommand{\esac}{\end{array}\right.\toggletrue{no_cases}}

\newcommand{\Poly}{\C[X]}

\newcommand{\rad}{\operatorname{rad}}

\renewcommand{\a}{\mathfrak{a}}
\newcommand{\fa}{\mathfrak{a}}

\newcommand{\fm}{\mathfrak{m}}

\newcommand{\rank}{{\operatorname{rank}}}

\newcommand{\tuple}[2]{(#1_{1},\ldots,#1_{#2})}
\newcommand{\drank}{d_{\operatorname{rank}}}
\newcommand{\maple}[2]{\operatorname{Map}_{\le}(#1,#2)}

\newcommand{\Mat}{\operatorname{Mat}}
\newcommand{\Np}{\mathbb N_+}

\newcommand{\Bij}{\operatorname{Bij}}

\begin{document}
\title[Almost commuting matrices...]{Almost commuting matrices with respect to the rank metric}
\author{Gábor Elek}
\author{Łukasz Grabowski}
\address{Department of Mathematics And Statistics, Fylde College, Lancaster University, Lancaster, LA1 4YF, United Kingdom}

\email{g.elek@lancaster.ac.uk}  
\email{lukasz.grabowski@lancaster.ac.uk}

\thanks{G.E. was partially supported
by the ERC Consolidator Grant ``Asymptotic invariants of discrete groups,
sparse graphs and locally symmetric spaces'' No. 648017. Ł.G. was partially supported by the ERC Starting Grant ``Limits of Structures in Algebra and Combinatorics'' No. 805495}

\begin{abstract}
We show that if  $A_1,A_2,\ldots, A_n$ are square matrices, each of them is  either unitary or self-adjoint, and they almost commute with respect to the rank metric, then one can find commuting matrices $B_1$, $B_2$, $\ldots$, $B_n$ that are close to the matrices $A_i$ in the rank metric.
\end{abstract}
\maketitle

\setcounter{tocdepth}{1}
\tableofcontents

\section{Introduction}

Recently there has been a considerable amount of research devoted to the following family of questions: suppose that square matrices $A$ and $B$ fulfil some relation ``approximately''. Can we then perturb $A$ and $B$ so that the resulting matrices $A'$ and $B'$ actually fulfil the relation in question? Let us make it more precise by reviewing some historical and more recent examples.

We start with the most famous one. Paul Halmos \cite{Hal} posed the following problem, known since as the \emph{Halmos problem}: Let $\de >0$, and suppose that $A$ and $B$ are self-adjoint matrices of norm $1$. Can we find $\eps>0$ such that if the operator norm of $AB-BA$ is at most $\eps$ then there exist self-adjoint matrices $A'$ and $B'$ such that $A'B' = B'A'$ and such that the operator norms of $A'-A$ and of $B'-B$ are at most $\de$? 

 An affirmative answer to this question was given by Huaxin Lin~\cite{HL} (see also \cite{FR} 
and~\cite{Has}). On the other hand, Voiculescu  proved that for integers 
$d\ge 7$ there exist $d\times d$ unitary matrices $U_d$, $V_d$ such 
that
\begin{itemize}
\item $\| U_d V_d-V_d U_d\|=|1-e^{2\pi i/d}|$, and
\item for any pair $A_d,B_d$ of commuting $d\times d$  matrices we have 
$$
\|U_d-A_d\|+\|V_d-B_d\|\ge \sqrt{2-|1-e^{2\pi i/d}|}-1.
$$
\end{itemize}

In other words, in the original Halmos problem, if we replace the assumption that $A$ and $B$ are self-adjoint with the assumption that $A$ and $B$ are unitary, then the answer is negative, even if we do 
not demand that the nearby commuting matrices $A'$ and $B'$ should be 
unitary. Furthermore, counterexamples were found by Davidson~\cite{Dav} if we 
ask about three or more almost commuting self-adjoint matrices.

A similar question had previously been asked by Rosenthal \cite{Ros}, 
where the ``closeness'' and ``almost commutativity'' of the matrices 
were defined using the normalised Hilbert-Schmidt norm in place of the operator norm. Affirmative 
answers to this version of the Halmos Problem were given for arbitrarily  
large finite families of normal operators by various 
authors \cite{HW},\cite{FS},\cite{Gle}. 

More recently the analogous question was studied in~\cite{AP1} for 
permutations and the Hamming distance. Arzhantseva and Paunescu showed 
the following result, which was a direct motivation for the 
investigations presented in this article. For every $\de>0$ there 
exists $\eps>0$ such that if $A$ and $B$ are permutations such that the 
normalised Hamming distance between $AB$ and $BA$ is at most $\eps$ 
then we can find permutations $A'$ and $B'$ such that $A'B' = B'A'$ and 
the normalised Hamming distances between $A$ and $A'$, as well as $B$ and $B'$, 
are both bounded by $\de$. The corresponding result is true also 
for an arbitrary finite number of permutations.

In this paper we study the analogous question for the \emph{rank metric}. 
We refer to~\cite{AP2} and the references therein for the background and 
motivation for studying rank metric, and here we only state the 
definitions. The set of natural numbers is $\N := \{0,1,\ldots\}$ and 
we let $\Np := \{1,2\ldots\}$. For $d \in \Np$ let $\Mat(d)$ be the set 
of all $d\times d$ square matrices with complex coefficients. Finally, for $A 
\in \Mat(d)$ we let  $\rank(A) := \frac{\dim_\C (\im(A))}{d}$. This 
norm defines a metric on $\Mat(d)$ in a usual way, i.e.~$\drank(A,B) := 
\rank(A-B)$.

Our main aim in this note is to show the following theorem.

\begin{theorem}\label{tmain} 
For every $\eps > 0$ and $n\in \Np$ there exists $\de > 0$ such that 
for all $d\in \Np$ we have the following. If $A_1, 
A_2,\dots A_n \in \Mat(d)$ are matrices,  each of them is either unitary or self-adjoint, and for 
all $1\le i,j \le n$ we have $\rank(A_iA_j-A_jA_i)\le \delta$, then 
there exist commuting matrices $B_1, B_2, \dots, B_n$ such that for 
every $1\le i \le n$  we have $\rank(A_i-B_i)\le \eps$. 
\end{theorem}

A more general statement will be presented in Theorem~\ref{t1}. 

\begin{remark}
It is natural to ask whether the matrices $B_1,\ldots, B_n$ can be 
taken to be ``of the same type'' as the matrices $A_1,\ldots,A_n$, 
e.g.~whether we can demand, say, the matrix $B_1$ to be unitary, provided that $A_1$ 
is unitary. We do not know the answer to this question.

We think that Theorem~\ref{tmain} likely stays true when $A_1,\ldots, 
A_n$ are allowed to be arbitrary normal invertible matrices. On the other hand, it 
would be interesting to find a counterexample when $A_1,\ldots, A_n$ 
are allowed to be arbitrary invertible matrices. 

\end{remark}

Becker, Lubotzky and Thom~\cite{blt} generalised the 
results from~\cite{AP1} to the context of finitely presented polycyclic 
groups, and showed that there are signi{-}ficant obstacles to generalise it further. We are able to prove some analogous results in the context of the rank metric. Let us make it precise now.

Let $\Ga$ be a finitely presented group with presentation 
$$
\langle \ga_1,\ldots, \ga_g | P_1(\ga_1,\ldots, \ga_g), \ldots, P_r(\ga_1,\ldots, \ga_g)\rangle,
$$
where $P_i$ are non-commutative monomials in $g$ variables (we allow negative exponents here).
. 

For a $k\times k$ matrix $B$ we denote with $\widehat B$ the operator on the vector space 
$\C^{\oplus\N}$ which acts as $B$ on the first $k$ basis vectors and is $0$ 
otherwise.

We will say that $\Ga$ is \emph{stable with respect to the rank metric} 
if for every $\de>0$ there exists $\eps>0$ such that the following holds. For all $d\in \N$ we have that if $A_1,\ldots, A_g$ are unitary $d\times d$ 
matrices with $\rank(P_i(A_1,\ldots, A_g)-\Id_d) \le \de$, then there exist $k\in \N$ and invertible $k\times k$  matrices 
$B_1,\ldots, B_g$ with $\dim\left(\im(\widehat{A_i}-\widehat{B_i})\right) \le \eps\cdot d$ and such that $P_i(B_1,\ldots, B_g) = \Id_k$ for all $i=1,\ldots, r$.

\begin{remarks}
\begin{enumerate}

\item Originally, we  have not worked with $\widehat{A_i}$ but rather with $A_i$ in the definition above. We thank Narutaka Ozawa for pointing out that it is more natural to take $\widehat{A_i}$. 

\item It is not hard to check (and we use it implicitly in the discussion above) that the property of being stable with respect to the rank metric does not depend on the choice of a finite presentation of the group $\Ga$.

\item Theorem~\ref{tmain} implies that the groups $\Z^k$, where $k=1,2,\ldots$, are stable with respect to the rank metric. We remark that there exist other natural notions of \emph{being stable with respect to the rank metric}: for example, we could demand the matrices $B_i$ to be unitary, or we could remove the assumption that the matrices $A_i$ are unitary. Thus, to avoid confusion, it might be useful to talk about, say, $(\cal A,\cal B)$-stability, where $\cal A =((\cal A_1,d_1),(\cal A_2,d_2),\ldots) $ is a sequence of monoids with metrics, and $\cal B = (\cal B_1,\cal B_2,\ldots)$ is a sequence of groups such that $\cal B_i \subset \cal A_i$. We refrain from doing this in this paper as all our results are about the stability with respect to the rank metric, as defined above.

\end{enumerate}
\end{remarks}

Perhaps the most interesting question which we cannot tackle at present is inspired by the results of~\cite{blt}: are polycyclic groups stable with respect to the rank metric? However, by using some of the ideas from~\cite{blt} we can show the following result.

Let $p$ be a prime number. Recall that Abels' group $A_p$ (see~\cite{Abels}) is the group of $4$-by-$4$ matrices of the form
$$
\begin{pmatrix}  
1 & \ast & \ast & \ast  \\
 & p^m & \ast  & \ast  \\
 & & p^n & \ast \\
& & & 1 
 \end{pmatrix},
$$
where $m,n\in \Z$, and where the stars are arbitrary elements of the ring $\Z[\frac{1}{p}]$ of rational numbers which can be written with a power of $p$ as the denominator. 

\begin{theorem}\label{abels}
For any prime number $p$ the Abels' group $A_p$ is not stable with respect to the rank metric.
\end{theorem}

\begin{remarks}\begin{enumerate} \item It is not difficult to show that if a finitely presented amenable group is stable with respect to the rank metric then it is residually linear. Thus, mimicking the question posed in~\cite{AP1}, one could ask whether every finitely presented linear amenable group is stable with respect to the rank metric. Since Abels' group is a solvable group of step 3 which is finitely presented and linear, Theorem~\ref{abels} gives a negative answer to this question.

\item Our proof of Theorem~\ref{abels} is based on an argument from~\cite{blt} used to show that the Abels' groups are not stable with respect to the Hamming distance. In fact, Theorem~\ref{abels} is a generalisation of that particular result from~\cite{blt}. We will present the proof of~Theorem~\ref{abels} in Section~\ref{sec-abels}. It is very self-contained and can also serve a minor role as an alternative exposition of one of the results of~\cite{blt} (the advantage of our proof of Theorem~\ref{abels} compared with the exposition in~\cite{blt} is somewhat smaller definitional overheads).
\end{enumerate}
\end{remarks}
\newcommand{\inter}{\operatorname{int}}

We would like to thank the referees for a very careful reading of this paper, and numerous helpful suggestions and corrections.

\section{The strategy of the proof and the general statement of Theorem~\ref{tmain}}

Let us very informally discuss the strategy of the proof of Theorem~\ref{tmain}. For 
simplicity let us assume that we are given two $d\times d$ matrices $A$ 
and $B$ which are almost commuting with respect to the rank metric.

First, we need to find a large subspace $W\subset \C^d$ and a 
decomposition $W = \bigoplus_{i=1}^N B_i$, such that each space $B_i$ 
has the following two properties:
\begin{enumerate}
\item there exists $R_i\in \N$, an ideal $\fa_i\subset \C[X,Y]$, and a 
linear embedding $\phi_i \colon B_i \to \C[X,Y]/\fa_i$ whose image 
consists of all elements of degree at most $R_i$, and

\item ``$B_i$ is almost invariant for the actions of $A$ and $B$''.
\end{enumerate}

Most of Section~\ref{sec-proof} is devoted to finding such $W$, culminating in 
Lemma~\ref{lemma-final}. This allows us to replace the original $A$ 
and $B$ with direct sums of multiplication operators in commutative 
algebras, restricted to ``balls in the algebras'', i.e.~to subspaces of 
polynomials with degree bounded by $R_i$.

The reduction of the proof of Theorem~\ref{t1} to 
finding such $W$ is described in Lemma~\ref{lemma-first-red}.

The property that $A$ and $B$ are either self-adjoint or unitary is 
used in two ways. The first use is controlling the nilpotent elements 
in the resulting commutative algebras. This is done in 
Lemma~\ref{lemma-reg}. While controlling the nilpotent elements greatly 
simplifies the proof, the authors believe it is not essential.

The second, more crucial, use is making sure that the subspace $W$ is 
large. Informally speaking, the assumption that $A$ and $B$ are either 
self-adjoint or unitary allows us to argue that if $W$ is small, then 
we can add some extra subspaces $B_i$ in the orthogonal complement of 
$W$ (see Lemma~\ref{lemma-bootstrap}). The argument is very similar to 
the ``Ornstein-Weiss trick'' (see \cite{OW}), and the assumption on $A$ 
and $B$ allows us to replace ``disjointedness'' with ``orthogonality''.

After finding $W$ we still need to consider the operators of 
multiplication by $X$ and $Y$ in $\C[X,Y]/\fa_i$ restricted to polynomials 
of degree  bounded by $R_i$. These two restrictions clearly almost commute, and we 
need to perturb them with  small rank operators to obtain commuting 
operators. 

In order to be able to carry out the Ornstein-Weiss trick in our 
setting, we make use of the effective Nullstellensatz (encapsulated in Theorem~\ref{cory-effective})  and the Macaulay 
theorem on growth in graded algebras (encapsulated in Corollary~\ref{corymac}). The final commutative algebra tool which we use  is the standard Nullstellensatz (Proposition~\ref{alghom}).

The effective Nullstellensatz (i) allows us to argue that the embeddings 
$\phi_i$ exist, i.e.~reduce ``the local situation to the commutative 
algebra'', and (ii) together with the assumption that $A$ and $B$ are 
either unitary or self-adjoint, it allows us to control the nilpotent 
elements in the resulting commutative algebras $\C[X,Y]/\fa_i$. It is 
used in Lemma~\ref{lemma-reg}.

The Macaulay theorem (i) allows us to argue that the complement of $W$ 
is small, and (ii) it allows us to argue that the commuting 
perturbations of multiplication operators in commutative algebras which 
we find, are indeed small rank perturbations. It is used in 
Lemmas~\ref{lemma-complicated} and~\ref{lemma-bootstrap}.

\subsection*{Definitions and the general statement}

Elements of $\Mat(d)$ will be called \emph{$d$-matrices}. Tuples 
of $d$-matrices will be called \emph{$d$-matrix tuples}, and will be 
denoted with curly letters, e.g.~$\cal A = (A_1,\ldots, A_n)$ and $\cal 
M = (M_1,\ldots, M_n)$.

For $a\in \Np$, the symbol $[a]$ denotes the set $\{1,2,\dots,a\}$, and 
we let $[0]$ denote the empty set. We say that a matrix tuple 
$\cal A = (A_1,\ldots, A_n)$ is  \emph{commuting} if for all $i, j \in 
[n]$  we have $A_{i}A_{j} - A_{j}A_{i} = 0$. More generally, for $\eps 
\ge 0$ we say that $\cal A$ is $\eps$-commuting if 
$$
    \max_{i,j\in [n]} \rank(A_{i}A_{j} - A_{j}A_{i}) \le  \eps.
$$
If $d\in\Np$ and  $\cal A = (A_1,\ldots, A_n)$, $\cal B = (B_1,\ldots, B_n)$ are two $d$-matrix tuples, then we let
$$
    d_{\rank} (\cal A, \cal B) := \max_{i\in [n]} \rank(A_i-B_i).
$$

Given a matrix $A$, we denote the adjoint of $A$ by $A^\ast$. We say 
that a $d$-matrix tuple $\cal M=(M_1,\dots, M_n)$ is 
\emph{$\sta$-closed} if for every $i\in [n]$ there exists $j\in [n]$ 
such that $M_i^\ast=M_j$.

Our general result is as follows.

\begin{theorem}\label{t1}
For every $\eps\ge 0$ and $n\in\Np$ there exists $\de \ge 0$ such that 
if 
$$
\cal A=(A_1,\ldots, A_n)
$$ 
is a $\sta$-closed $\de$-commuting matrix 
tuple then we can find a commuting matrix tuple $\cal B$ with 
$$
    d_\rank(\cal A, \cal B) \le \eps.
$$
\end{theorem}

Let us argue how to deduce Theorem~\ref{tmain} from Theorem~\ref{t1}. 
First, we note that if we replace the expression \emph{a $\sta$-closed 
$\de$-commuting matrix tuple} in the statement of Theorem~\ref{t1} by \emph{a $\de$-commuting 
matrix tuple such that each of the matrices $A_1,\ldots A_n$ is either 
self-adjoint or unitary} then we obtain the statement of Theorem~\ref{tmain}.

 But if  
$(A_1,\ldots, A_n)$ is any matrix tuple, then $(A_1,\ldots, A_n, 
A_1^\ast,\ldots, A_n^\ast)$ is a $\sta$-closed matrix tuple. As such, in 
order to deduce Theorem~\ref{tmain} from Theorem~\ref{t1} it is enough 
to prove the following proposition.

\begin{proposition}
For every $n\in \Np$, every $\de>0$ and every $d\in \Np$, we have that if
 $(A_1,\ldots, A_n)$ is a $\de$-commuting  $d$-matrix tuple and each of the matrices $A_1,\ldots, A_n$ is either unitary or self-adjoint, then the $d$-matrix tuple $(A_1,\ldots, A_n,A_1^\ast,\ldots, A_n^\ast)$ is $\de$-commuting as well. 
\end{proposition}

\begin{proof}
Using induction, it is enough to show that if $A$ is a $d$-matrix and $B$ is either a unitary or a self-adjoint $d$-matrix with $\rank(A,B) \le \de$ then also $\rank(A,B^\ast)\le \de$.

If $B$ is self-adjoint then there is nothing to prove. If $B$ is unitary then we will use the fact that $B^\ast = B^{-1}$.
 We let 
$$
W:= \ker (AB-BA),
$$
and by assumption we have $\dim(W) \ge 1-\de$. For  $v\in B(W)$ we can write $v= B(w)$ for some $w\in W$, hence we obtain that 
$$
B^{-1}A(v) = B^{-1}AB(w)= B^{-1}BA(w)= A(w).
$$ 
On the other hand we can write
$$
AB^{-1}(v) = AB^{-1}B(w) = A(w).
$$
This shows that $B(W)\subset \ker(AB^{-1}-B^{-1}A)$, finishing the proof because 
$$
\dim(B(W)) = \dim(W)\ge 1-\de.
$$
\end{proof}

\begin{remark} With a little bit more effort we could also deal with matrix tuples whose all elements are normal matrices with spectrum contained in the union of the real line and the unit circle.
\end{remark}

For the rest of the paper we fix a positive natural number $n$. From now on all matrix tuples will have length $n$.

\section{Commutative algebra preliminaries}

Let $\C$ be the field of complex numbers. The ring $\C[X_1,\ldots, X_n]$ 
will be denoted by $\C[X]$. Recall that an ideal $\fa \subset \Poly$ is 
\emph{radical} if for all $m\in \N_+$ and $f\in \Poly$ we have that 
$f^m \in \fa$ implies $f\in \fa$. By Hilbert's Nullstellensatz we have 
that $\fa = \bigcap \fm$, where the intersection is over all 
maximal ideals which contain $\a$. 

Given an arbitrary ideal $\fa$ we denote by $\rad(\fa)$ the 
\emph{radical of $\fa$}, i.e.~the radical ideal defined as $\rad(\fa) : 
= \{ f\in \Poly: \exists m \in \Z_+ \text{ with } f^m \in \fa\}$.

The next theorem follows from the effective Nullstellensatz of Grete 
Hermann \cite{MR1512302} and the Rabinowitsch trick (see e.g.~\cite[Theorem 1 and the corollary afterwards]{MR916719}). 

\begin{theorem}\label{cory-effective}
There exists an increasing function $K\colon \N \to \N$ such that we have the following properties. 
Let $f, f_1,\ldots, f_k\in \Poly$ be polynomials of degree at most $R$, and let $\fa$ be the ideal generated by $f_1,\ldots, f_k$.
\begin{enumerate}
\item If  $f \in \fa$  then there exist $h_1,\ldots, h_k\in \Poly$ such that
$$
    h_1f_1 + \ldots + h_k f_k = f
$$
and $\deg(h_if_i) \le K(R)$.

\item If $f\in \rad(\fa)$ then we can find $m\in \N$ and  $g_1,\ldots, g_k\in \Poly$ such that  
$$
    g_1f_1+\ldots  +g_kf_k =f^m
$$
and $\deg(g_if_i)\le K(R)$.
\end{enumerate} \qed
\end{theorem}

In the applications of this theorem we will implicitly use that $K(R)\ge R$.

For the next proposition we need to recall some definitions. A \emph{standard graded $\C$-algebra} is a $\C$-algebra $A$ together with a family of vector spaces $A_i$, $i\in \N$, such that 
\begin{enumerate}
\item $A_0 = \C$, $A = \oplus_{i\in \N} A_i$, 
\item $A$ is generated as a $\C$-algebra by finitely many elements of $A_1$,
\item for all $i,j\in \N$ we have $A_i A_j \subset A_{i+j}$.
\end{enumerate}

A \emph{filtration} on a $\C$-algebra $A$ is an ascending family $F_0\subset F_1\subset \ldots $ of linear subspaces of $A$ such that 
\begin{enumerate}
\item $F_0 = \C$, $A = \bigcup_{i\in \N} F_i$,
\item for all $i,j\in \N$ we have $F_i F_j \subset F_{i+j}$.
\end{enumerate}

\newcommand{\gr}{\operatorname{gr}}

Given an algebra $A$ with a filtration $F_i$, $i\in \N$,  we can associate to it a 
graded algebra $\gr(A)$ as follows. As a $\C$-vector space we let 
$\gr(A) := F_0 \oplus \bigoplus_{i>0} F_i/F_{i-1}$. We define the 
multiplication on $\gr(A)$ first on the elements of the form $a+F_i$ 
and $b + F_j$, where $i,j\ge 0$, $a\in F_{i+1}$, $b\in F_{j+1}$, by the 
formula $(a + F_i) \cdot (b + F_j)  := ab+ F_{i+j+1}$. In general we 
extend this multiplication to all of $\gr(A)$ by $\C$-linearity.

\begin{remark} The reason why $\gr(A)$ is not always a \emph{standard} 
graded algebra is that it may happen not to be generated by the 
elements of $F_1/F_0$. This may be the case even if $A$ is  generated 
by finitely many elements of $F_1$ as a $\C$-algebra.

For example, let $A := \C[X_1]$, let $F_0=\C$, let $F_1$ be the vector 
space of polynomials of degree at most $1$, and finally for $i\ge 2$ 
let $F_i$ be the vector space of polynomials of degree at most $2i-1$. 
In this case we have $(X_1+ F_0)^2 = X_1^2 + F_1$, and therefore $(X_1+F_0)^3 =X_1^3 + F_2$, i.e.~$(X_1+F_0)^3$ is equal to $0$ in $\gr(A)$. 

In fact, it is not difficult to construct examples where $\gr(A)$ fails to be 
finitely-generated, even when $A$ is generated by finitely many elements of $F_1$.
\end{remark}

We say that $F_i$ is a \emph{standard} filtration on $A$ if the associated graded algebra $\gr(A)$ is standard. 

The following is a consequence of Macaulay's theorem \cite{MR1576950}. We will use the exposition from \cite[Section 5]{2017arXiv170301761E}.

\begin{proposition}\label{promac}
Let $A$ be a $\C$-algebra with a standard filtration $F_i$. Then for every $i>0$ we have
$$
\dim_\C(F_i/ F_{i-1}) < \frac{\dim_\C(F_1)}{i} \dim_\C(F_{i-1}).
$$
\end{proposition}

\begin{proof}

For a natural number $k$ and a real number $x$ we let ${x \choose k}$ denote the number $\frac{1}{k !} \cdot x(x-1)\cdot\ldots\cdot(x-k+1)$.

Let us fix $i>0$. After applying \cite[Theorem 5.10]{2017arXiv170301761E} to the standard graded algebra $\gr(A)$ we obtain the following. Let $x$ be the unique real number such that $x \ge i-1$ and 
$$
    \dim_\C (F_{i-1} ) = { x \choose i-1}.
$$
Then we have that
$$
    \dim_\C (F_i)  \le {x+1 \choose i}.
$$

In particular, we also obtain that
\begin{equation}\label{gog}
    \frac{ \dim(F_i)}{ \dim(F_{i-1})} \le \frac{{x+1 \choose i}}{{ x \choose i-1}} = 
    \frac{x+1}{i}.
\end{equation}

On the other hand, we have  $\dim(F_1) = {\dim(F_1)\choose 1}$, and the function  $X \mapsto {X \choose k}$ is  increasing for $X \ge k-1$ (see \cite[Lemma 5.6]{2017arXiv170301761E}). Thus if we apply \cite[Theorem 5.10]{2017arXiv170301761E}  $i-2$ times starting with $\dim(F_1) = {\dim(F_1)\choose 1}$,  then we obtain
$$
    \dim_\C (F_{i-1})  \le {\dim(F_1)  +i-2 \choose i-1},
$$
implying that $x < \dim(F_1) + i - 1$. Together with \eqref{gog}, this shows that 
$$
\frac{ \dim(F_i)}{ \dim(F_{i-1})} < \frac{i+\dim(F_1)}{i},
$$
so the proposition follows since 
$$
\dim(F_i) = \dim(F_i/ F_{i-1}) + \dim(F_{i-1}).
$$
\end{proof}

\begin{definition}\label{fil}
Given an ideal $\fa\subset \Poly$, we introduce a standard filtration $F^\fa_i$, $i\in \N$, on $\Poly /\fa$ by defining $F^\fa_i$ to be the space of all those elements of $\Poly/\fa$ which can be written as $f +\fa$ with $\deg(f) \le i$.
\end{definition}

Applying Proposition~\ref{promac} to the filtration $F^\fa_i$, we obtain the following Corollary.
\begin{cory}\label{corymac}
Let $\fa\subset \Poly$ be an ideal. Then for any $i>0$ we have
$$
\dim_\C(F^\fa_i/ F^\fa_{i-1}) \le \frac{n}{i} \dim_\C (F^\fa_{i-1}).
$$\qed

\end{cory}

We now proceed to derive some properties of multiplication operators restricted to the spaces $F^\fa_i$, $i\in \N$. We start with a simple consequence of Hilbert's Nullstellensatz. When for some $k\in \Np$ we consider the space $\C^k$ as a $\C$-algebra, it is meant to be with the pointwise multiplication.

\begin{proposition}\label{alghom}
Let $\fa\subset \C[X]$ be an ideal and let $V\subset \Poly$ be a finite dimensional $\C$-linear subspace with the property that $V \cap \rad(\fa) = \{0\}$. Then there exists a surjective algebra homomorphism $\si \colon \Poly \to \C^{\dim(V)}$ which is injective on $V$ and such that $\fa \subset \ker(\si)$.
\end{proposition}

\begin{proof}
We prove by induction on $\dim(V)$ the following statement: There exist 
distinct maximal ideals $\fm_1,\ldots, \fm_{\dim(V)}$ such that for all $i$ we have $\rad(\fa)\subset \fm_i$ and 
$$
V\cap \fm_1\cap \ldots\cap \fm_{\dim(V)} = \{0\}.
$$

For the case $\dim(V) = 1$ let us first choose a non-zero element $v\in V$. Now since $v\notin \rad(\fa)$ and $\rad(\fa)$ is equal to an intersection of maximal ideals, we can find a maximal ideal $\fm_1$ such that $\rad(\fa) \subset \fm_1$ and $v\notin \fm_1$.

Let us therefore assume that we know the inductive statement when \\ $\dim(V) = k$ for some $k$ and let us fix $V$ such that $\dim(V) = k+1$. Let $W\subset V$ be a $k$-dimensional subspace. By the inductive assumption we can find $\fm_1,\ldots, \fm_{k}$ such that $W\cap\fm_1\cap \ldots\cap \fm_k =\{0\}$. Thus the intersection $V\cap\fm_1\cap \ldots\cap \fm_k$ is at most one-dimensional. It cannot be zero-dimensional because the composition $V \into \Poly\ \to \Poly/ (\fm_1\cap \ldots\cap \fm_k)$ has a non-trivial kernel, since the Chinese remainder theorem implies that the right-hand side is isomorphic to $\C^k$.

Thus the intersection $V\cap\fm_1\cap \ldots\cap \fm_k$ is one-dimensional. Let $v$ be a non-zero element of $V\cap\fm_1\cap \ldots\cap \fm_k$. Since $v\notin \rad(\fa)$ and $\rad(\fa)$ is equal to an intersection of maximal ideals, we can find a maximal ideal $\fm_{k+1}$ such that $\rad(\fa) \subset \fm_{k+1}$ and $v\notin \fm_{k+1}$. Thus $V\cap\fm_1\cap \ldots\cap \fm_k \cap \fm_{k+1} = \{0\}$,
finishing the proof of the inductive claim.
    
Now we can define $\si$ as being the quotient map $\Poly \to \Poly/(\fm_1\cap \ldots\cap \cap \fm_{\dim(V)})$. This finishes the proof. 
\end{proof}

\begin{definition}\label{mumaps}
Given an ideal $\fa \subset \Poly$ we denote by $\mu^\fa_i\colon \Poly/\fa \to \Poly/\fa$ the linear map defined by  $\mu^\fa_i(f) = X_i\cdot f$.
\end{definition}

\begin{corollary}\label{cory-alghom}
Let $\fa\subset \C[X]$ be an ideal and let $R\in \N$ be such that 
$$
F^\fa_R\cap \left(\rad(\fa)/\fa\right) = \{0+\fa\}.
$$
Then there exist simultaneously diagonalisable linear maps 
$$
    M_1,\ldots, M_n\colon F^\fa_R \to F^\fa_R
$$  
such that for $v\in F^\fa_{R-1}$ and all $i=1,\ldots, n$ we have $M_i(v) = \mu^\fa_i(v)$.
\end{corollary}

\begin{proof}
Let $d := \dim(F^\fa_R)$, and let $f_1,\ldots, f_d\in \Poly$ be such that $f_i+\fa$ is a basis of $F^\fa_R$. Let $V\subset \Poly$ be the linear span of the elements $f_i$. Let us observe that $V\cap \rad(\fa) = \{0\}$. Indeed, if $f\in V\cap \rad(\fa)$ then $f+\fa \in F^\fa_R\cap \rad(\fa)/\fa$ and so by assumption we see that $f \in \fa$.

Hence by the previous proposition we can find a surjective algebra homomorphism 
$$
    \si\colon \Poly/\fa \to \C^d
$$ 
such that $\si$ is injective on $F^\fa_R$. Let $\tau\colon \C^d \to F^\fa_R$ be the unique linear isomorphism such that for $v\in F^\fa_R$ we have $\tau(\si(v)) = v$. Since $\si$ is surjective, it follows that for all $v\in \C^d$ we have $\si(\tau(v))=v$.

For $v\in F^\fa_R$ let us define
$$
    M_i(v): = \tau(\si(X_i\cdot v)).
$$

If $v\in F^\fa_{R-1}$ then $X_i\cdot v\in F^\fa_R$ and so $M_i(v) = \tau(\si(X_i\cdot v) = X_i\cdot v = \mu^\fa_i(v)$. Thus in order to finish the proof we only need to check that the maps $M_i$ are simultaneously diagonalisable. 

Let $e_1,\ldots, e_d$ be the standard basis of $\C^d$. In particular $\tau(e_1),\ldots, \tau(e_d)$ is a basis of $F^\fa_R$, and we claim that for every $i\in \{1,\ldots, n\}$ we have that the vectors $\tau(e_j)$, $j=1,\ldots, d$, are eigenvectors for $M_i$. Indeed, first we note that for every $i$ and $j$ we have that $\si(X_i)\cdot e_j$ is a multiple of $e_j$, and so we can define numbers $\la_{ij}\in \C$ by the formula 
$$
    \si(X_i)\cdot e_j  = \la_{ij} e_j.
$$
Now we can write  
\begin{align*}
    M_i(\tau(e_j)) &= \tau\Big(\si(X_i\cdot \tau(e_j))\Big) = \tau\Big(\si(X_i) \cdot \si(\tau(e_j))\Big) \\
    &= \tau(\si(X_i)\cdot e_j) = \tau(\la_{ij}e_j) = \la_{ij} \tau(e_j),
\end{align*}
finishing the proof.
\end{proof}

\section{Proof of Theorem~\ref{t1}}\label{sec-proof}

We will first prove several lemmas. The first lemma, informally speaking, allows us to deduce Theorem~\ref{t1} provided that we can construct large subspaces by ``growing balls around points''. To make it precise we state a few definitions. 

Given two positive natural numbers $a,b$ we let $\map(a,b)$ be the set of all maps from $[a]$ to $[b]$, and furthermore we let $\maple a b := \bigcup_{ i=0}^a \map(i,b)$.

Let $W$ be a $\C$-vector space and let $\cal M = (M_1,\ldots, M_n)$ be a tuple of endomorphisms of $W$. Given $R\in\N$ and $\al\in \map(R,n)$ we let 
$$
    \cal M_\al := M_{\al(1)}\cdot \ldots \cdot M_{\al(R)}.
$$
Note that the unique element of $\map(0,n)$ is the empty set. Our convention is that $\cal M_\emptyset$ is the identity map.

Given $w\in W$  we let $B_{\cal M}(w,R)$ to be the linear span of the vectors $\cal M_\al (w)$, 
where $\al\in \maple R n$. We will call $B_{\cal M} (w,R)$ the \emph{$R$-ballspace for $\cal M$ with root $w$}. If $\cal M$ is clear from the context, then we denote $B_{\cal M}( w,R)$ simply with $B(w, R)$.

Recall from Definition~\ref{mumaps} that given an ideal $\fa \subset \Poly$ and $i\in [n]$, we denote by $\mu^\fa_i\colon \Poly/\fa \to \Poly/\fa$ the linear map  defined by  $\mu^\fa_i(f +\fa) = (X_i\cdot f) +\fa$. Note that the $R$-ballspace for $\tuple {\mu^\fa}n$ with root $1+\fa \in \Poly/\fa$ is equal to $F^\fa_R$.

In general we will say that $B_{\cal M}(w,R)$ is \emph{regular} if there exists an ideal $\fa \subset \Poly$ and a linear isomorphism $\phi\colon B_{\cal M}(w,R) \to F^\fa_R$ such that the following two conditions hold.

\begin{enumerate}
\item For every $v \in B_{\cal M}(w,R-1)$ and $i\in [n]$ we have that $\phi(v) \in F^\fa_{R-1}$ and $\phi(M_i(v)) =\mu^\fa_i(\phi(v))$.
\item If for some $f\in \Poly$ and $m \in \Np$ we have $f +\fa \in F^\fa_R$ and  $f^m\in \fa$ then $f\in \fa$. In other words we have $F^\fa_R \cap (\rad(\fa)/\fa) = \{0+\fa\}$.
\end{enumerate}

Let $d\in \N$ and let $\tuple Mn$ be a $d$-matrix tuple.  Given $R\in \N$ and a subspace $W\subset \C^d$ we say that $W$ is an $R$-multi-ballspace if there exist $w_1,\ldots, w_k\in W$ and natural numbers $R_1,\ldots, R_k$ with $R_j\ge R$ such that 
\begin{enumerate}
\item the ballspaces $B(w_j,R_j)$ are regular, and
\item $W$ is equal to the direct sum $\bigoplus_{j=1}^k B(w_j, R_j)$.
\end{enumerate}

The \emph{roots} of such $W$ are the points $\{w_1,\ldots, w_k\}$. 

If all elements of a matrix tuple $\cal A$ can be diagonalised simultaneously, then $\cal A$ will be called \emph{simultaneously diagonalisable}. Clearly, if $\cal A$ is a simultaneously  diagonalisable tuple then it is also a commuting tuple.

\begin{lemma}\label{lemma-first-red}
For every $\eps>0$ there exists $R\in \N$ and $\de>0$ such that the following holds. Suppose that $d\in \N$, let $\cal M$ be a $d$-matrix tuple, and let $W\subset \C^d$ be an $R$-multi-ballspace with $\dim(W) \ge (1-\de) \cdot d$.

Then there exists a simultaneously diagonalisable $d$-matrix tuple $\cal A$ such that 
$$
    \drank(\cal M , \cal A) \le \eps.
$$
\end{lemma}

\begin{proof}
Let $R$ be such that $\frac{n}{R}< \frac{\eps}{2}$ and let $\de$ be such that $\de<\frac{\eps}{2}$. Let $\cal M = \tuple Mn$ be the $d$-matrix tuple, let  $w_1,\ldots w_k\in \C^d$ and let $R_1,\ldots, R_k\in \N$ be such that $R_i\ge R$ and such that the ballspaces $B(w_j,R_j)$ are regular and $W = \bigoplus_{j=1}^k B(w_j, R_j)$.  

We need to find a  simultaneously diagonalisable tuple $\cal A = \tuple An$ such that 
$\drank(\cal M, \cal A) \le \eps$.

For every $j=1,\ldots, k$, let $\phi_j\colon B(w_j,R_j) \to F^{\fa_j}_{R_j}$ be the linear isomorphism witnessing the regularity of the ballspace $B(w_j,R_j)$. 

By Corollary~\ref{cory-alghom}, for every $j=1,\ldots, k$ we can  find maps 
$M_{ij}\colon F^{\fa_j}_{R_j}\to F^{\fa_j}_{R_j}$, where $i=1,\ldots, n$, such that the maps $M_{1j}, \ldots, M_{nj}$ pairwise commute, are simultaneously diagonalisable, and for $v\in B(w_j, R_j-1)$ we have $M_i(v) = \phi_j^{-1}\cdot M_{ij}\cdot \phi_j(v)$.

Let us fix a projection $\pi \colon \C^d \to  \bigoplus_{j=1}^k B(w_j,R_j)$, and for $i=1,\ldots, n$ let
$$
    A_{i} := \left(\bigoplus_j\phi_j^{-1}\cdot M_{ij}\cdot \phi_j \right) \cdot \pi.
$$ 

It is clear that the maps $A_1,\ldots, A_n$ are simultaneously diagonalisable. Also for every $v\in \bigoplus_{j=1}^k B(w_j,R_j-1)$ we have $A_i(v) = M_i(v)$, so 
\begin{multline*}
\drank(\cal M, \cal A) \le
\\
\le \frac{1}{d} \left(\dim(\ker(\pi)) + \sum_{j=1}^k \big(\dim (B(w_j,R_j)) - \dim (B(w_{j},R_{j}-1))\big)\right).
\end{multline*}

Since the ballspaces $B(w_j,R_j)$ are regular, by Proposition~\ref{promac} we have 
\begin{align*}
    \dim( B(w_j,R_j)) - \dim (B(w_{j},R_{j}-1)) &\le \frac{n}{R_j} \cdot \dim (B(w_j,R_j-1)) 
\\
&\le \frac{n}{R} \cdot \dim (B(w_j,R_j)).
\end{align*}

Hence we see that
$$
\sum_{j=1}^k  \big(\dim (B(w_j,R_j)) - \dim (B(w_{j},R_{j}-1))\big) \le \frac{n}{R} \sum_{j=1}^k \dim (B(w_j,R_j)) < \frac{\eps}{2} \cdot d.
$$
Thus altogether we have
$$
\drank(\cal M , \cal  A) < \frac{\eps}{2} + \de \le \eps,
$$
finishing the proof.
\end{proof}

Given $r,d\in \N$, a $d$-matrix tuple $\cal M = \tuple Mn$, and $v \in \C^d$, we say that $\cal M$ is \emph{$r$-commutative} on $v$ if for any $i\le r$, any $\al \in \map(i,n)$ and any permutation $\si\colon [i] \to [i]$  we have 
$$
        \cal M_\al(v) = \cal M_{\al\circ \si}(v).
$$ 

If $W\subset\C^d$ then we say that $\cal M$ is \emph{$r$-commutative on $W$ } if for every $v\in W$ we have that $\cal M$ is $r$-commutative on $v$.

\begin{lemma} \label{lemma-pair-reduction}
For every $R\in \N$ and $\eps>0$ there exists $\eta>0$ such that if $d\in \N$ and $\cal M$ is an $\eta$-commuting $d$-matrix tuple, then there exists a subspace $W\subset \C^d$ such that $\cal M$ is $R$-commutative on $W$ and $\dim(W) \ge d(1-\eps)$.
\end{lemma}
\begin{proof}
For $k\in\N$ we let $\Bij(k)$ be the set of all bijections of the set $[k]$.
Let us prove by induction on $R$ that for every $\eps>0$ there exists $\eta>0$ such that if $d\in \N$ and $\cal M$ is an $\eta$-commuting $d$-matrix tuple, then
$$
    \dim \left( \bigcap_{\substack{\al\in \map(R, n) \\ \si\in\Bij(R)}} \ker(\cal M_\al -\cal M_{\al\circ \si})\right) \ge d(1-\eps).
$$

When $R=2$, we can set $\eta := \frac{\eps}{n^2}$. Indeed, if $\cal M$ is an 
$\eta$-commuting tuple then by definition for $i,j\in [n]$ we have $\dim(\ker([M_i,M_j])) \ge (1-\eta)d$. And so we have 
$$
    \dim \left( \bigcap_{i,j\in [n]} \ker([M_i,M_j])\right) \ge (1-n^2\eta)d = (1-\eps)d.
$$ 

Thus let us assume that we have shown the inductive statement for some $R$ and let us prove it for $R+1$. Let us fix $\eps>0$ and let $\eta$ be given by the inductive assumption for $\frac{\eps}{n+1}$. Thus given an $\eta$-commuting tuple $\cal M$ we obtain a subspace $W\subset \C^d$ such that $\dim(W) \ge d(1-\frac{\eps}{n+1})$ and $\cal M$ is $R$-commutative on $W$, i.e.~for any $w\in W$, any $\al\in \map( k, n)$ with $k\le R$ and any permutation $\si \colon [k] \to [k]$ we have 
$$
    \cal M_\al (w) = \cal M_{\al\circ \si} (w).
$$ 

Let us define $V := W \cap \bigcap_{i=1}^n M_i^{-1}(W)$. Clearly $\dim(V) \ge d(1-\eps)$.

Now let $\be \in \map (R+1, n)$, let $\tau\colon [R+1]\to [R+1]$ be a permutation, and let $v\in V$. Let $i\in [n]$ be such that for some $2\le j_1 \le R+1$ we have $\be(j_1) =i$ and for some $2 \le j_2 \le R+1$ we have $\beta(\tau(j_2)) = i$. We can find such $i$ because $R+1\ge 3$.

Since in particular $v\in W$, we can find $\ga\in \map(R,n)$ and a permutation  $\rho\colon [R]\to [R]$  such that 
$$
        \cal M_{\be}(v) = \cal M_{\ga}\cdot M_i (v)
$$
and
$$
        \cal M_{\be\circ \tau}(v) = \cal M_{\ga\circ \rho}\cdot M_i (v).
$$
Since $M_i(v) \in W$, we have
$$
    \cal M_{\ga\circ \rho}\cdot M_i (v) =\cal M_{\ga}\cdot M_i (v),
$$
which finishes the proof.
\end{proof}

\begin{definition}\label{def-pair}
If $r\in \N$, $\eps>0$, and $A$ and $B$ are subspaces of $\C^d$, then we say that $(A,B)$ is an $(r,\eps)$-pair for the $d$-tuple $\cal M=(M_1,M_2,\dots,M_n)$ if
\begin{enumerate}
\item $A\subset B$ and  $\dim(B/A) \le \eps\cdot d$, and 
\item for every $\al \in \maple rn$ and $v\in A$ we have $\cal M_\al (v) \in B$.
\end{enumerate}
\end{definition}

\begin{lemma}\label{lemma-complicated}
Let $\eps>0$, let $r,R,d\in \N$ with $n\le r<R$, and let $\cal M = \tuple Mn$ be a $\sta$-closed $d$-matrix tuple. Let $(A_1,B_1)$ be an $(R,\eps)$-pair and let $W = \bigoplus_{j=1}^k B(w_j, R_j)$ be an $R$-multi-ballspace for $\cal M$ contained in $B_1$. Furthermore let us assume that the ballspaces $B(w_j, R_j+r)$ are regular for all $j$. Finally, let $B_2$ be the orthogonal complement of $W$ in $B_1$. 

Then there exists a subspace $A_2 \subset A_1\cap B_2$ such that $(A_2, B_2)$ is an $(r, \eps+ \frac{n}{R}2^r)$-pair.

\end{lemma}

\begin{proof}
 Let $V\subset \C^d$ be the space spanned by the ballspaces $B(w_j, R_j+r)$, $j=1,\ldots, k$. Let $V^\perp\subset \C^d$ be the space orthogonal to $V$ and let $A_2 =A_1\cap V^\perp$.

Let us show that $\dim(B_2/A_2) \le d(\eps + \frac{n}{R}2^r)$. By basic linear algebra, it is easy to check that $\dim(B_2/A_2)$ is bounded from above by 
\begin{equation}\label{apa}
    \dim(B_1/A_1) + \dim(V) - \dim(W).
\end{equation}
We can bound \eqref{apa} by
$$
    \dim(B_1/A_1) + \sum_{j=1}^k \big(\dim (B(w_j,R_j+r)) - \dim( B(w_j,R_j))\big).
$$
By Corollary~\ref{corymac}, the quantity above is at most 
$$
    \eps\cdot d + \sum_{j=1}^k \dim (B(w_j,R_j))\left(\left( 1+ \frac{n}{R}\right)^r -1\right) \le \eps\cdot d + d \cdot \frac{n}{R} \cdot 2^r, 
$$
where we use the inequality $(1+x)^r \le 1+(2^r-1)x$, valid for $x\in [0,1]$ and $r\ge 1$.
 Therefore we obtain that  $\dim(B_2/A_2) \le d(\eps + \frac{n}{R}2^r)$.

Thus to finish the proof we only need to show that for  $x\in A_2$ and $\al \in \map(q,n)$ with $q\le r<R$ we have $\cal M_\al(x) \in B_2$. In other words, we need to show that if $x \in A_1$ is orthogonal to $B(w_j, R_j+r)$ then $\cal M_\al (x)$ is orthogonal to $B(w_j,R_j)$.

Indeed, let $w\in B(w_j,R_j)$. Since $\cal M$ is $\sta$-closed, we have $\cal M_\al^\ast(w) \in B(w_j,R_j+r)$ and hence
$$
    \langle \cal M_\al (x), w \rangle  = \langle x, \cal M_\al^\ast(w) \rangle = 0,
$$ 
finishing the proof.
\end{proof}

Let $\C\langle Y_1,\ldots, Y_n\rangle$ be the ring of polynomials in $n$ non-commuting variables, and let $\pi \colon \C\langle Y_1,\ldots, Y_n\rangle \to \Poly$ be the algebra homomorphism such that $\pi(Y_i) = X_i$. Let $c\colon \Poly \to \C\langle Y_1,\ldots, Y_n\rangle$ be the unique $\C$-linear map such that for any $\al \in \map(d,n)$ with $\al(1)\le \al(2)\le \ldots \le \al(d)$ we have 
$$
    c\circ\pi( Y_{\al(1)}\ldots Y_{\al(d)}) = Y_{\al(1)}\ldots Y_{\al(d)}.
$$

In other words, the map $c$ allows us to treat commutative polynomials as non-commutative ones, by fixing an order on the variables. Given a matrix tuple $\cal M = (M_1,\ldots, M_n)$ and $f\in \Poly$, we define $f(\cal M)$ to be the matrix $c(f)(M_1,\ldots, M_n)$.

Let us define a ${}^\ast$-operation on $\C\langle Y_1,\ldots Y_n\rangle$ in the following way. For any $\al \in \map(d,n)$ and $a\in \C$ we define
$$
   \left( a \cdot Y_{\al(1)}\cdot \ldots \cdot Y_{\al(d)} \right)^\ast := \ov{a} \cdot Y_{\al(d)}\cdot \ldots \cdot Y_{\al(1)},
$$
and we extend this definition to arbitrary elements of $\C\langle Y_1,\ldots Y_n\rangle$ by linearity.

For $f\in \Poly$ we define $f^\ast(\cal M) := (c(f)^\ast)(M_1,\ldots, M_n)$. The following simple observation will be used without reference.

\begin{lemma}
For any matrix tuple $\cal M$ and any $f\in \Poly$ we have that the matrices $f(\cal M)$  and $f^\ast(\cal M^\ast)$ are adjoint to each other. \qed
\end{lemma}

We will also need the following lemma.

\begin{lemma}\label{lemma_simple2}
Let $f\in \Poly$ and let $k\in \N$. Let $\cal M$ be a $\sta$-closed $d$-matrix tuple and let $x\in \C^d$ be such that $\cal M$ is $(2k \deg(f))$-commutative at $x$. Then
$$
    \left[f^\ast(\cal M^\ast) f(\cal M)\right]^k (x) = \left[f^\ast(\cal M^\ast)\right]^k\left[f(\cal M)\right]^k(x)
$$
\end{lemma}

\begin{proof}
For $f\in \Poly$ let us define $\bar f$ to be the polynomial which arises from $f$ by conjugating the coefficients.  Then the left-hand side is equal to
$$
    \left[\bar f(\cal M^\ast) f(\cal M)\right]^k(x),
$$
and the right-hand side is equal to 
$$
\left[\bar f (\cal M^\ast)\right]^k\left[f(\cal M)\right]^k(x).
$$
These two expressions are clearly equal if $\cal M$ is $(2k \deg(f))$-commutative at $x$.
\end{proof}

Recall that $K(R)$ is a function defined in Theorem~\ref{cory-effective}. 

\begin{lemma}\label{lemma-reg}
Let $d\in \Np$ and let $\cal M = (M_1,\ldots, M_n)$ be a $\sta$-closed $d$-matrix tuple. Let $R\ge 0$ and let $v\in \C^d$ be such that $\cal M$ is $(2 K(R))$-commutative at $v$. Then the ballspace $B(v, R)$ is regular.
\end{lemma}

\begin{proof}
Given $\al \in \map(q,n)$, we define $X_\al \in \Poly$ to be the monomial
$$
    X_\al := X_{\al(1)}\ldots X_{\al(q)}.
$$

Let $P\subset \Poly$ be defined as follows. We let $f\in P$ if and only if $\deg(f) \le R$  and $f(\cal M)(v) =0$. Let $\fa$ be the ideal generated by $P$. Let us define a map $\phi \colon B(v,R) \to  \Poly/\fa$ as follows: 
$$
    \phi( \cal M_\al (v)) := X_{\al}+\fa.
$$

Let us check that $\phi$ is well-defined. For this let us assume that 
$$
 \sum_{\al \in \maple R n} s_\al \cal M_\al(v) =  \sum_{\al \in \maple R n } t_\al \cal M_\al(v),
$$
where $s_\al, t_\al \in \C$. But then 
$$
\sum_{\al \in \maple R n } (s_\al - t_\al) M_\al(v) = 0,
$$
and therefore $\sum_{\al \in \maple R n } (s_\al-t_\al) X_\al \in P$. In particular we get that
$$
\sum_{\al \in \maple R n} s_\al X_\al + \fa =  \sum_{\al \in \maple R n } t_\al X_\al + \fa,
$$
which shows that $\phi$ is well-defined.

Now let us see that $\phi$ is injective. Indeed suppose that 
$$
\phi \left( \sum_{\al} s_\al \cal M_\al(v) \right) = 0,
$$
where $\al$ runs through the elements of $\maple R n$, and $s_\al \in \C$. Then \\ $\sum_\al s_\al X_\al \in \fa$, and so we can find  $f_i \in P$  and $h_i \in \Poly$ with $\deg(h_if_i) \le K(R)$ such that 
$$
\sum_{i=1}^k h_if_i  = \sum_\al s_\al X_\al.
$$

But since $\cal M$ is $K(R)$-commutative at $v$, we have
$$
\sum_{i=1}^k h_i(\cal M) f_i(\cal M)(v)  = \sum_\al s_\al \cal M_\al (v).
$$

The left-hand side is equal to $0$ since $f_i\in P$, and so we see that \\  $\sum_{\al} s_\al \cal M_\al (v) = 0$. This finishes the proof of injectivity of $\phi$.

Since clearly the image of $\phi$ is equal to $F^\fa_R$, it remains to prove that 
$$
F^\fa_R \cap (\rad(\fa)/\fa) = \{0+\fa\}.
$$

By Theorem~\ref{cory-effective}, if $f \in \Poly$ is such that $\deg(f) \le R$ and $f\in \rad(\fa)$ then we can find $m \in\N$, elements $f_i \in P$ and $g_i\in \Poly$ with  $\deg(g_if_i) \le K(R)$, such that 
$$
f^m = \sum_{i=1}^k g_i f_i,
$$

Since $\cal M$ is $K(R)$-commutative at $v$, we have
$$
    0 = \sum_{i=1}^k g_i(\cal M) f_i(\cal M)(v) = f(\cal M)^m(v),
$$
and so by $2K(R)$-commutativity, and since $R\le K(R)$, we also have
$$
0 =f(\cal M)^m(v) = f^\ast(\cal M^\ast) ^m f(\cal M)^m(v) = \left(f^\ast(\cal M^\ast)f(\cal M)\right)^m (v).
$$

In particular we can define $t$ to be the smallest positive integer such that 
$$
    \left(f^\ast(\cal M^\ast)f(\cal M)\right)^t (v) = 0.
$$

We will show $t=1$. By way of contradiction, let us consider two cases: first let us assume that $t$ is even and equal to $2l$ for some $l\ge 1$. 

By Lemma~\ref{lemma_simple2}, we have
\begin{equation}\label{eq-adjoint}
    0 = f^\ast(\cal M^\ast)^{2l} f(\cal M)^{2l}(v) = (f^\ast(\cal M^\ast)f(\cal M))^{2l}(v). 
\end{equation}
Therefore, we also have
\begin{align*}
    0 &= \langle (f^\ast(\cal M^\ast) f(\cal M))^{2l}(v) , v\rangle \\
      &= \langle (f^\ast(\cal M^\ast) f(\cal M))^{l}(v), (f^\ast(\cal M^\ast) f(\cal M))^{l}(v) \rangle.
\end{align*}
This shows that $(f^\ast(\cal M^\ast) f(\cal M))^{l}(v) = 0$, which contradicts the minimality of $t$. 

In the second case, let us assume that $t$ is odd and equal to $2l+1$ for some $l\ge 1$. We proceed in a similar fashion. By Lemma~\ref{lemma_simple2} we have that $\left( f^\ast(\cal M^\ast) f(\cal M)\right)^{2l+1}(v) = 0$. Hence, we also have
\begin{align*}
  0 &= \langle \left( f^\ast(\cal M^\ast) f(\cal M)\right)^{2l+1}(v), f^\ast(\cal M^\ast) f(\cal M)(v)\rangle \\
    &= \langle \left( f^\ast(\cal M^\ast) f(\cal M)\right)^{l+1}(v), \left( f^\ast(\cal M^\ast) f(\cal M)\right)^{l+1}(v) \rangle,
\end{align*}
and since $l+1< t$, we obtain a contradiction exactly as in the first case.

Thus all in all we have showed that $f^\ast(\cal M^\ast)f(\cal M) (v) = 0$. Since $f^\ast(\cal M^\ast)$ and $f(\cal M)$ are adjoint to each other, we also have that $f(\cal M)(v) = 0$. This shows that $f\in P$, and hence $f\in \fa$, which finishes the proof.
\end{proof}

\begin{lemma}\label{lemma-bootstrap}
Let $R\in \N$, $d\in \N$, let $\cal M$ be a $\sta$-closed $d$-matrix tuple, let $A\subset \C^d$, and let us assume that $\cal M$ is $2K(2R)$-commutative on $A$. Then there exist $k\in \N$ and $w_1,\ldots, w_k\in A$ such that the $R$-ballspaces $B_{\cal M}(w_i,R)$ are regular, 
pairwise orthogonal, and we have that
\begin{equation}\label{todo6}
    \sum_{i=1}^k \dim\big(B_{\cal M} (w_i, R)\big)  \ge  \frac{1}{e^n}\cdot \dim(A),
\end{equation}
where $e=2.71...$.
\end{lemma}
\begin{proof}
Note that by Lemma~\ref{lemma-reg} all $R$-ballspaces with roots in $A$ are regular. Let us consider the subset $Q$ of $A^{\oplus\N}$ which consists of those tuples $(w_1,\ldots, w_k)$ such that the ballspaces $B_{\cal M}(w_1,R), \ldots, B_{\cal M}( w_k,R)$ are pairwise orthogonal to each other.

Let $(w_1,\ldots, w_k)\in Q$ be a tuple for which the number
$$
    \sum_{i=1}^k \dim(B_{\cal M}(w_i,R))
$$
is maximal. By Lemma~\ref{lemma-reg}, it is enough to show that $\sum_{i=1}^k \dim(B_{\cal M}(w_i,R)) \ge \frac{1}{e^n}\cdot \dim (A)$.  Consider the vector space $V$ spanned 
by the ballspaces $B_{\cal M}(w_i,2R)$. By Lemma~\ref{lemma-reg}, the ballspaces $B_{\cal M}(w_i,2R)$ are regular, and so by Corollary~\ref{corymac} we have that
\begin{multline*}
    \dim (B_{\cal M}(w_i, 2R)) \le \\
\le   \left(1 +\frac{n}{2R}\right)\left(1 +\frac{n}{2R-1}\right)
        \ldots \left(1 + \frac{n}{R+1}\right) \dim( B_{\cal M}(w_i, R)),
\end{multline*}
which easily implies that
$$
    \dim( B_{\cal M}(w_i, 2R)) \le e^n \dim( B_{\cal M} (w_i, R)).
$$
This shows that 
\begin{equation}\label{done43}
    \dim(V)  \le e^n \dim\left(\bigoplus_{i=1}^k B_{\cal M} (w_i, R)\right).
\end{equation}

Let us observe that if $x\in A$ is orthogonal to $V$ then $B_{\cal M}(x,R)$ is orthogonal to the space
$\bigoplus_{i=1}^k B_{\cal M} (w_i, R)$. Indeed, since $\cal M$ is $\sta$-closed, for any $\al \in \maple R n$ and any $w\in B_{\cal M}(w_i,R)$ we have that $\cal M_\al^\ast(w)\in B_{\cal M}(w_i,2R)$. It follows that 
$$
    \langle \cal M_\al(x),w\rangle = \langle x , \cal M_\al^\ast (w) \rangle = 0.
$$

But by the maximality of $(w_1,\ldots, w_k)$, the above shows that there are no points in $A$ orthogonal to $V$, so in fact we have $A\subset V$. In particular, we have $\dim(V) \ge \dim( A)$, and hence by \eqref{done43} we have 
$$
    \dim( A) \le e^n \dim\left(\bigoplus_{i=1}^k B_{\cal M} (w_i, R)\right),
$$  
finishing the proof.
\end{proof}

The final lemma which we need for the proof of Theorem~\ref{t1} is an "Ornstein-Weiss type" lemma.

\begin{lemma}\label{lemma-final}
For every $\de>0$ and $r\in \N$ there exists $\eta>0$ such that 
if $d \in \N$,  and $\cal M$ is an $\eta$-commuting $\sta$-closed $d$-matrix tuple, then there exists an $r$-multiball $W\subset \C^d$ for $\cal M$, such that $\dim(W) \ge (1-\de)\cdot d$.
\end{lemma}

\begin{proof} Let us fix $\de>0$ and $r\in \N$. Let us first fix  a positive integer $k$ such that $k(1-\frac{1}{e^n})^k < \frac{\de}{4}$, and then let us choose $\eps \in (0,\min(\frac{\de}{4}, \frac{\de}{2k},\frac{1}{e^n}))$ and natural numbers $r_0> r_1>\ldots > r_k = r$ such that  for $i=0,\ldots, k-1$ we have
$$
\eps +\frac{n}{r_i} 2^{r_{i+1}} < (1-\frac{1}{e^n})^k.
$$

By Lemma~\ref{lemma-pair-reduction}, we can fix $\eta$ to be such that if $\cal M$ is $\eta$-commuting then there exists a subspace $S\subset \C^d$ such  that $\dim(S) \ge (1-\eps)d$ and $\cal M$ is $2K(2(r_0+r_1))$-commutative on $S$. 

Let $\bar d := (1-\eps)d$. We will prove by induction on $i$ the following statement: for every $i=1,\ldots, k$ there exists $g(i)\in \N$, roots $w_1,\ldots, w_{g(i)}\in S$ and radii $R_1,\ldots, R_{g(i)}$ with $r_i \le R_j \le r_0$ for all $j= 1,\ldots,g(i)$, such that the balls $B_{\cal M} (w_j, R_j)$ are pairwise orthogonal and 
$$
\sum_{j=1}^{g(i)} \dim (B_{\cal M} (w_j,R_j)) \ge \bar d\left(1- i(1-\frac{1}{e^n})^i\right) -d \cdot \frac{i\de}{2k}.
$$
This will be enough to  finish the proof because for $i=k$ the right hand side above is equal to
$$
(1-\eps)d (1-k(1-\frac{1}{e^n})^k) - d \frac{\de}{2} > (1-\frac{\de}{4})(1-\frac{\de}{4})\cdot d -\frac{\de}{2} d  > (1-\de)\cdot d
$$


For $i=1$ the inductive claim is implied by Lemma~\ref{lemma-bootstrap}. Suppose that the inductive claim holds for some $i\in \{1,\ldots,k-1\}$ and let us prove it for $i+1$.

Let $W_i = \oplus_{j=1}^{g(i)} B_{\cal M} (w_j,R_j)$ and let $W_i^\perp$ be the orthogonal complement of $W_i$ in $\C^d$. Since $\cal M$ is $2K(2(r_0+r_{i+1}))$-commutative on $S$, we have that all the ballspaces 
$$
B_{\cal M} (w_j, R_j + r_{i+1})
$$ are regular, 
and so we can apply Lemma~\ref{lemma-complicated} for the $(r_i,\eps)$-pair $(S, \C^d)$ and the $r_i$-multiballspace $W_i$.

As a result we obtain a subspace $S_i \subset S\cap W_i^\perp$ such that $(S_i,W_i^\perp)$ is an $(r_{i+1}, \eps + \frac{n}{r_i}2^{r_{i+1}})$-pair. 

Now, by Lemma~\ref{lemma-bootstrap} we obtain $g(i+1) \in \N$ and roots $w_{g(i)+1}, w_{g(i)+2},\ldots, \\ w_{g(i+1)}   \in S_i$, such that the ballspaces $B_{\cal M}(w_{g(i)+s}, r_{i+1})$, $s=1,\ldots, g(i+1)-g(i)$, are regular, pairwise orthogonal, and 
$$
    \sum_{s=1}^{g(i+1)-g(i)} \dim( B_{\cal M}(w_{g(i)+s}, r_{i+1})) \ge \frac{1}{e^n}\dim(S_i).
$$ 
Since $(S_i,W_i^\perp)$ is an $(r_{i+1}, \eps + \frac{n}{r_i}2^{r_{i+1}})$-pair, we have 
$$
\dim (S_i) \ge \dim( W_i^\perp) - d(\eps + \frac{n}{r_i}2^{r_{i+1}}) 
$$
and so 
\begin{align*}
    \sum_{s=1}^{g(i+1)-g(i)} \dim( B_{\cal M}(w_{g(i)+s}, r_{i+1}))  &\ge \frac{1}{e^n} \dim (W_i^\perp) - \frac{d(1-\frac{1}{e^n})^k}{e^n}\\
    &\ge  \frac{1}{e^n} \dim( W_i^\perp) - \bar d(1-\frac{1}{e^n})^{i+1}.
\end{align*}
Therefore we have also
\begin{align*}
\sum_{j=1}^{g(i+1)} \dim( B_{\cal M} (w_j,R_j)) &\ge \dim (W_i) + \frac{1}{e^n} \dim (W_i^\perp)  - \bar d(1-\frac{1}{e^n})^{i+1} \\
    &= d - \dim(W_i^\perp) + \frac{1}{e^n} \dim (W_i^\perp)  - \bar d(1-\frac{1}{e^n})^{i+1} \\
    &\ge \bar d(1 -(1-\frac{1}{e^n})^{i+1}) - (1-\frac{1}{e^n})\dim (W_i^\perp).
\end{align*}
By the inductive assumption, we have $\dim (W_i^\perp) \le \bar d\cdot i(1-\frac{1}{e^n})^i + d\cdot i\frac{\de}{2k}+ \eps d$, so 
altogether we have
\begin{align*}
\sum_{j=1}^{g(i+1)} \dim (B_{\cal M} (w_j,R_j)) &\ge \bar d(1 -(1-\frac{1}{e^n})^{i+1}) - \bar d\cdot i(1-\frac{1}{e^n})^{i+1}  -d (i\frac{\de}{2k} +\eps)
\\
&= \bar d(1 -(i+1)(1-\frac{1}{e^n})^{i+1}) -d (i\frac{\de}{2k} +\eps)
\\
&\ge \bar d(1 -(i+1)(1-\frac{1}{e^n})^{i+1}) -d (i\frac{\de}{2k} +\frac{\de}{2k}),
\end{align*}
which is the inductive statement we wanted to show. Hence the lemma follows.
\end{proof}

We have now everything in place to prove Theorem~\ref{t1}.

\begin{proof}[Proof of Theorem~\ref{t1}]
Let us fix $\eps>0$. By Lemma~\ref{lemma-first-red}, we can fix $R\in \N$ and $\eta>0$ such that if $\cal A$ is a $d$-matrix tuple for some $d\in \N$, and    
$W\subset \C^d$ is an $R$-multi-ballspace for $\cal A$ with $\dim(W) \ge (1-\eta)d$, then we can find a commuting  $d$-matrix tuple $\cal B$ such that
$$
    d_\rank(\cal A, \cal B) \le \eps.
$$

However by Lemma~\ref{lemma-final}, we can find $\de>0$ such that if $d\in \N$ and $\cal A$ is a $\sta$-closed $\de$-commuting tuple then there exists an $R$-multi-ballspace $W\subset \C^d$ for $\cal A$ such that 
$$
\dim(W) \ge (1-\eta)d.
$$ 
This finishes the proof.
\end{proof}

\section{Abels' group is not stable with respect to the rank metric}\label{sec-abels}

We finish the article with the following proof.
\begin{proof}[Proof of Theorem~\ref{abels}]

The centre $Z(A_p)$ of $A_p$ is the group of matrices of the form
$$
\begin{pmatrix}  
1 & 0 & 0 & \ast  \\
 & 1 & 0  & 0  \\
 & & 1 & 0 \\
& & & 1 
 \end{pmatrix},
$$
isomorphic to $\Z[\frac{1}{p}]$. Consider the central subgroup $H$ of the elements of the form
$$
\begin{pmatrix}  
1 & 0 & 0 & x  \\
 & 1 & 0  & 0  \\
 & & 1 & 0 \\
& & & 1 
 \end{pmatrix},
$$
where $x\in \Z$. 

Let $\De$ be the quotient group $A_p/H$ and let $\pi\colon A_p \to \De$ be the quotient map. Let $\ga_1,\ldots, \ga_g$ be  generators of $A_p$ and let $P_1,\ldots, P_r$ be noncommutative monomials (possibly with negative exponents) such that
$$
\langle \ga_1,\ldots, \ga_g | P_1(\ga_1,\ldots, \ga_g), \ldots, P_r(\ga_1,\ldots, \ga_g)\rangle,
$$
is a presentation of $A_p$, and let $H$ be such that for all $i$  we have that $P_i$ has length at most $H$.  Let  $F_1,F_2,\ldots\subset \De$ be a sequence of F{\o}lner sets in $\De$. Let $S\subset \De$ be the set of those elements which can be represented as words of length at most $H$ in the elements $\pi(\ga_1),\ldots, \pi(\ga_g)$ and their inverses. For $i\in \Np$ let $\inter(F_i)$ be the subset of those $f\in F_i$ such that for all $s\in S$ we have $sf\in F_i$.

For $i\in \Np$ and $j=1,\ldots,g$ let  $A_i^j$ be a permutation of $F_i$ which is equal to $\pi(\ga_j)$ on $\inter(F_i)$ (there is in general no unique such permutation). In what follows we will think of $A_i^j$ as permutation matrices -- in particular they are unitary matrices.

Since $F_i$ is a F{\o}lner sequence, we have, for any $k\in \{1,\ldots, r\}$ that
$$
\rank(P_k(A_i^1,\ldots, A_i^g)-\Id_{|F_i|}) \xrightarrow[i\to \infty]{} 0,
$$
since the left-hand side is bounded from above by $1-\frac{|\inter F_i|}{|F_i|}$.

By way of contradiction, let us assume that $A_p$ is stable with respect to the rank metric. It follows that we can find $g$ sequences of invertible matrices $B_i^1,\ldots, B_i^g$ with 
$$
    \frac{1}{|F_i|}\dim\big(\im(\widehat{A_i}-\widehat{B_i})\big) \xrightarrow[i\to \infty]{} 0
$$  
and such that $P_i(B_1,\ldots, B_g) = 1$ for all $i=1,\ldots, r$. 
In particular for each $i=1,2,\ldots$ we get a representation $\rho_i\colon A_p \to GL(n_i,\C)$ for suitable $n_i \in \N$. 

Now let $t$ be a generator of $H$. Since $t$ is a central element, we have that each eigenspace of $\rho_i(t)$ is preserved under the action of $A_p$. Let $V_i\subset \C^{n_i}$ be the eigenspace of $\rho_i(t)$ corresponding to the eigenvalue $1$, i.e.~the set of all $v\in \C^{n_i}$ such that $\rho_i(t)(v) = v$.
 Note that $\frac{\dim(V_i)}{n_i}\xrightarrow[i\to\infty]{} 1$

We have representations $\bar \rho_i$ of $A_p/H$ on all the spaces $V_i$. Now let $K\subset Z(A_p)$ be the subgroup of $Z(A_p)$ of elements of the form $\frac{n}{p}$, where $n\in \Z$, and let $\bar K$ be the image of $K$ in $Z(A_p)/H$. Since $\bar K$  is finite, we may assume that $i$ is big enough so that $\rank(\bar \rho_i(\ga)-1) \ge \frac{1}{2}$ for all $\ga\in\bar K \setminus \{e\}$.

But for every element $\eta\in Z(A_p)\setminus H$ there exists $n\in \Np$ such 
that $\eta^n \in K\setminus H$. It follows that for every $\eta\in Z(A_p)\setminus H$ we have that $\rho_i(\eta)$ does not act as the identity on $V_i$. This shows that $\bar\rho_i$ is injective on $Z(A_p)/H$. 

But $\De$ is finitely-generated, and hence $\bar\rho_i(\De)$ is a finitely-generated linear group, which by Malcev's theorem~\cite{mal} shows that $\De$ is residually finite. But the abelian group $Z(A_p)/H\subset \De$ is not residually finite (see~\cite{blt} for a short argument), which is a contradiction. This finishes the proof.
\end{proof}

\end{document}